\documentclass[11pt]{amsart}
\usepackage[utf8]{inputenc}
\usepackage{multirow}
\usepackage{amsmath,amsthm, amssymb,comment}
\usepackage{mathtools}\usepackage{xcolor}
\newtheorem{theorem}{Theorem}[section]

\newtheorem{lemma}[theorem]{Lemma}
\newtheorem{corollary}[theorem]{Corollary}

\theoremstyle{definition}

\newtheorem{remark}[theorem]{Remark}
\newtheorem{example}[theorem]{Example}

\begin{document}

\title{Condition Numbers of Hessenberg Companion Matrices}
\author{Michael Cox} 
\address{Unit 202 - 133 Herkimer Street, Hamilton, ON, L8P 2H3, Canada}
\email{michael.m.cox@outlook.com}

\author{Kevin N. Vander Meulen}
\thanks{Research of Vander Meulen was supported in part 
by NSERC Discovery Grant
2022-05137.}
\address{Department of Mathematics\\
Redeemer University College, Ancaster, ON, L9K 1J4, Canada}
\email{kvanderm@redeemer.ca}

\author{Adam Van Tuyl}
\thanks{Research of Van Tuyl was supported in part by NSERC Discovery Grant 2019-05412}
\address{Department of Mathematics and Statistics\\
McMaster University, Hamilton, ON, L8S 4L8, Canada}
\email{vantuyl@math.mcmaster.ca}

\author{Joseph Voskamp}
\thanks{Research of Voskamp was supported in part by NSERC USRA 504279.}
\address{Department of Mathematics and Statistics\\
McMaster University, Hamilton, ON, L8S 4L8, Canada}
\email{voskampj@mcmaster.ca}

\date{\today}


\begin{abstract}
The Fiedler matrices are a large class of companion
matrices that include the well-known Frobenius companion
matrix. The Fiedler matrices are part of a larger class
of companion matrices that can be characterized with
a Hessenberg form. In this paper, we demonstrate that
the Hessenberg form of the Fiedler companion matrices
provides a straight-forward way to compare the  
condition numbers of these matrices. We also show
that there are other companion matrices which can 
provide a much smaller condition number than any
Fiedler companion matrix. We finish by exploring the condition number of a class of matrices obtained from perturbing a Frobenius companion matrix while preserving the characteristic polynomial. 
\end{abstract}

\maketitle

\section{Introduction}

The Frobenius companion matrix is a template that provides a matrix with a prescribed characteristic polynomial. 
More recently, it was discovered that the Frobenious companion matrix belongs to a larger class of Fiedler companion matrices~\cite{F}, which in turn is a subset of the intercyclic companion matrices~\cite{EKSV}. Other recent templates include  nonsparse companion matrices \cite{DFGV} and  generalized companion matrices \cite{GSSV}. 

The Frobenius companion matrix is employed in algorithms that use matrix methods to determine roots of polynomials, but this matrix is not always well-conditioned \cite{dTDP}. Recent work \cite{dTDP} has explored under what circumstances other Fielder companion matrices can have a better condition number than the Frobenius matrix, with respect to the Frobenius norm. After covering background details in Section~\ref{sec:back},  we use a Hessenberg characterization of the Fiedler companion matrices in Section~\ref{sec:Hess} to provide a concise argument for the condition number of a Fielder companion matrix. The characterization allows us to avoid 
dealing with the particular permutation in Fiedler's construction of companion matrices~\cite{F}, as well as
 associated concepts around consecutions and inversions developed in~\cite{dTDP}.
In Section~\ref{sec:stripe}, we provide some examples of non-Fiedler companion matrices that demonstrate that there are intercyclic companion matrices that have a smaller condition number than any Fielder companion matrix for some specific polynomials. In Section~\ref{sec:gen}, we provide a method for constructing a generalized companion matrix that, in some cases, can improve on the condition number of any Fiedler companion matrix.


\section{Technical definitions and background}\label{sec:back}

In this section we recall the relevant 
background on companion matrices and condition numbers that will be
required throughout the paper.

Let $n \geq 2$ be an integer
and $p(x)=x^n+c_{n-1}x^{n-1}+c_{n-2}x^{n-2}+\cdots+c_0$.
A \emph{companion matrix}
to $p(x)$ 
is  an $n\times n$ matrix $A$ 
over $\mathbb{R}[c_0,\ldots,c_{n-1}]$ 
such that the characteristic polynomial
of $A$ is $p(x)$.
A \emph{unit sparse companion matrix}
to $p(x)$ 
is a companion 
matrix $A$ 
that has $n-1$ entries equal to one, 
$n$ variable entries $-c_0,\ldots,-c_{n-1}$, 
and the remaining 
$n^2-2n+1$ entries equal to zero.
The unit sparse companion matrix of the form
$$
\begin{bmatrix}
0 & 1 & 0 & \cdots & 0 & 0 \\
0 & 0 & 1 & \cdots & 0 & 0 \\
0 & 0 & 0 & \cdots & 1 &  0 \\
\vdots & \vdots & \vdots & \cdots & 0 & 1 \\
-c_0 & -c_1 & -c_2 & \cdots& -c_{n-2} & -c_{n-1} \end{bmatrix}
$$
is called the \emph{Frobenius companion matrix} of $p(x)$. Sparse companion matrices have also been called intercyclic companion matrices due to the structure of the
digraph associated with
the matrix (see \cite{VV} and \cite{EKSV} for details).

The matrices in Figure~\ref{fig:excom} are examples of 
unit sparse companion
matrices to 
$p(x)=x^4+c_3x^3+c_2x^2+c_1x+c_0$. The first 
matrix in Figure \ref{fig:excom} is a Frobenius companion matrix.  The matrices
in Figure~\ref{fig:gencompmat} are also companion
matrices to $p(x)$, but they are not unit
sparse since not every nonzero variable entry is the negative of a single coefficient of $p(x)$. 
Note that in the last matrix,
the value of $a$ can be any real number;
when $a=0$, then this matrix becomes a unit
sparse companion matrix.
\begin{figure}
$$\left[\begin{array}{rrrr}
		0&1&0&0\\
		0&0&1&0\\
		0&0&0&1\\
		-c_0&-c_1&-c_2&-c_3
		\end{array}\right], \quad
		\left[\begin{array}{rrrr}
		0&1&0&0\\
		0&-c_3&1&0\\
		0&-c_2&0&1\\
		-c_0&-c_1&0&0
		\end{array}\right], \quad
		\left[\begin{array}{rrrr}
		0&1&0&0\\
		-c_2&-c_3&1&0\\
		0&0&0&1\\
		-c_0&-c_1&0&0
		\end{array}\right].
		$$
\caption{Some $4\times 4$ unit sparse companion matrices.} \label{fig:excom}
\end{figure}

\begin{figure}
$$		\left[\begin{array}{rrrr}
		0&1&0&0\\
		-c_2&0&1&0\\
		-c_1+c_3c_2&0&-c_3&1\\
		-c_0&0&0&0
		\end{array}\right], \quad
		\left[\begin{array}{rrrr}
		-c_3&1&0&0\\
		 0  &0 &1&0\\
		-c_1+c_3c_2&-c_2&0&1\\
		 -c_0&0&0&0
		\end{array}\right], \quad
		\left[\begin{array}{rrrrr}
     -c_3    &1&0&0\\
     -c_2+a   &0&1&0\\
     -c_1+ac_3 &-a&0&1\\
     -c_0    &0&0&0
     \end{array}
     \right].
		$$
\caption{Some $4\times 4$  companion matrices.} \label{fig:gencompmat}
\end{figure}

\begin{figure}
$$		\left[\begin{array}{rrrr}
		-c_3&1&0&0\\
		-c_2&0&1&0\\
		-c_1&0&0&1\\
		-c_0&0&0&0
		\end{array}\right], ~~
\left[\begin{array}{rrrr}
		-c_3&-c_2&-c_1&-c_0\\
		1&0&0&0\\
		0&1&0&0\\
		0&0&1&0
		\end{array}\right], ~~
\left[\begin{array}{rrrr}
		0&0&0&-c_0\\
		1&0&0&-c_1\\
		0&1&0&-c_2\\
		0&0&1&-c_3
		\end{array}\right]. 
$$\caption{Some companion matrices equivalent to the 
$4\times 4$ Frobenius companion matrix.}\label{fig:eqFrb}
\end{figure}

Since matrix transposition and permutation similarity does not affect the characteristic polynomial, nor the set of nonzero entries in a matrix, we call two companion matrices \emph{equivalent} if one can be obtained
from the other via transposition and/or  permutation similarity.
The matrices in Figure~\ref{fig:eqFrb} are equivalent to the $4 \times 4$ Frobenius 
companion matrix. Note that if $A$ and $B$ are equivalent matrices, then the multiset of entries in any row of $A$ is exactly the multiset of entries of some row or column of $B$. No two matrices from Figures~\ref{fig:excom} and~\ref{fig:gencompmat}
are equivalent (assuming $a\neq 0$).

Fielder \cite{F} introduced a class of companion matrices 
that are constructed as
a product of
certain block diagonal matrices. In particular, let
$F_0$ be a diagonal matrix with diagonal
entries $(1,\ldots,1,-c_0)$
and for $k =1,\ldots,n-1$, let
$$
F_k = \left[\begin{array}{ccc} I_{n-k-1}&O&O\\O&T_k &O\\ O&O& 
I_{k-1}\end{array}\right] 
~~\mbox{with}~~  T_k=\left[\begin{array}{cc} -c_k&1\\1&0
\end{array}\right].
$$
Fiedler showed (see \cite[Theorem 2.3]{F})
that the product of these $n$ matrices, in any order,
will produce a companion matrix of $p(x)
=x^n+c_{n-1}x^{n-1}+c_{n-2}x^{n-2}+\cdots+c_0$.  
Consequently, given any permutation
$\sigma = (\sigma_0,\sigma_2,\ldots,\sigma_{n-1})$
 of $\{0,1,2,\ldots,n-1\}$, we say that 
$F_\sigma=F_{\sigma_0}F_{\sigma_1}\cdots F_{\sigma_{n-1}}$  is a 
\emph{Fiedler companion matrix}. The Frobenius companion matrix is a 
Fiedler companion matrix
since the Frobenius companion matrix is equivalent to
$F_0F_1\cdots F_{n-1}$,
as noted in~\cite{F}.  

In~\cite{EKSV} it was demonstrated that every unit sparse 
companion matrix is equivalent to a 
unit lower Hessenberg matrix, as summarized in 
Theorem~\ref{thm:sparse}. Note that, 
for
$0\leq k\leq n-1$, 
the $k$-th \emph{subdiagonal} of a matrix $A=[a_{ij}]$ consists of the entries $\{a_{k+1,1}, a_{k+2,2},\ldots,a_{n,n-k}\}$. The $0$-th subdiagonal is usually
called the main diagonal of a matrix.

\begin{theorem}\cite[Corollary 4.3]{EKSV} \label{thm:sparse}
Let $p(x) = x^n + c_{n-1}x^{n-1} + c_{n-2}x^{n-2} + 
\cdots + c_{1}x + c_0$ be a polynomial 
over $\mathbb{R}$ with $n\geq 2$. 
Then $A$ is an $n \times n$ unit sparse companion matrix to $p(x)$ if
and only if $A$ is equivalent to a unit lower Hessenberg matrix 
\begin{equation}\label{mat:sparseForm}
C  =  \left[ \begin{array}{ccccc}
\multirow{2}{*}{$\mathbf{0}$} & \multicolumn{2}{c|}{\multirow{2}{*}{$I_{m}$}} & \multicolumn{2}{c}{\multirow{2}{*}{$O$}} \\
                              & \multicolumn{2}{c|}{}                        & \multicolumn{2}{c}{} \\ \hline
\multicolumn{3}{c|}{\multirow{3}{*}{$R$}} & \multicolumn{2}{c}{\multirow{2}{*}{$I_{n-m-1}$}} \\
\multicolumn{3}{c|}{}                        & \multicolumn{2}{c}{} \\
\multicolumn{3}{c|}{}                        & \multicolumn{2}{c}{\mathbf{0}^T} \\
\end{array} \right]
\end{equation}
for some $(n-m) \times (m+1)$ matrix $R$ with $m(n-1-m)$ zero entries, 
such that $C$ has $-c_{n-1-k}$ on its $k$-th subdiagonal, 
for $0\leq k \leq n-1$.
\end{theorem}

Note that in \eqref{mat:sparseForm},
the unit lower Hessenberg matrix $C$
always has $C_{n,1}=-c_0$
and $R_{1,m+1}=-c_{n-1}.$
Given this Hessenberg characterization 
of the unit sparse companion matrices, one can deduce the 
corresponding inverse matrix if $c_0 \neq 0$.

\begin{lemma}
\label{lem:inverse}
\cite[Section 7]{VV}
Let $p(x) = x^n + c_{n-1}x^{n-1} + c_{n-2}x^{n-2} + 
\cdots + c_{1}x + c_0$ be a polynomial 
over $\mathbb{R}$ with $n\geq 2$. 
Suppose that $C$ is a unit lower Hessenberg companion matrix to $p(x)$ 
as in \eqref{mat:sparseForm}. 
Assuming $c_0 \neq 0$, if
$$ 
\renewcommand{\arraystretch}{1.5}
C = \left[
\begin{array}{c|c|c}
	\mathbf{0} & I_m & O \\ \hline
	\mathbf{u} & H & I_{n-m-1} \\\hline 
	{-c_0} & \mathbf{y}^T & \mathbf{0}^T \\
\end{array}
\right],~~\mbox{for some}~~  \mathbf{u},\mathbf{y}, H, ~~\mbox{then}~~
C^{-1} = \left[
\begin{array}{c|c|c}
	\frac{1}{c_0}\mathbf{y}^T & \mathbf{0}^T & {-\frac{1}{c_0}} \\ \hline
	I_m & O & \mathbf{0} \\ \hline 
 {-\frac{1}{c_0}\mathbf{u}\mathbf{y}^T - H} 
 & I_{n-m-1} & \frac{1}{c_0}\mathbf{u}
\end{array}
\right].
$$ 
\end{lemma} 

Throughout this
paper, we use the 
\emph{Frobenius norm}  
of an
$n \times n$ matrix
$A = [a_{i,j}]$ given by
$$|| A || = \sqrt{\sum_{i,j} a_{i,j}^2}.$$
\begin{remark}\label{rem:same}
If $A$ and $B$ are both unit sparse companion
matrices to the same polynomial $p(x)$, then
it follows that $||A|| = ||B||$ since $A$
and $B$ have exactly the same entries.
Furthermore, if $A = PBP^T$ for some
permutation matrix $P$, then $A^{-1}$ and
$B^{-1}$ also have the same entries,
and hence $||A^{-1}|| = ||B^{-1}||$.
\end{remark}
The {\it condition number}
of $A$, denoted $\kappa(A)$, is defined to be
$$\kappa(A)  = || A||\cdot ||A^{-1}||.$$
Remark~\ref{rem:same}  implies the following lemma.  
\begin{lemma}\label{equivalentcondition}
If $A$ and $B$ are equivalent companion matrices, 
then $\kappa(A) = \kappa(B)$.
\end{lemma}


\section{Condition numbers of Fiedler matrices via 
the Hessenberg characterization}\label{sec:Hess}

The condition numbers of Fiedler companion matrices were
first calculated by de Ter\'an, Dopico, and P\'erez
\cite[Theorem 4.1]{dTDP}. In this section we demonstrate how 
a characterization of Fielder companion
matrices via unit lower Hessenberg matrices, as given by Eastman,
{\it  et al.} \cite{EKSV},  provides an efficient way to
obtain the condition numbers for Fiedler companion
matrices.   Our approach avoids the
use of the consecution-inversion structure sequence,
described in \cite[Definition 2.3]{dTDP}, which was used
in the original computation of these numbers.

The following theorem gives a
characterization of the Fielder
companion matrices in terms of unit
lower Hesenberg matrices.

\begin{theorem}\cite[Corollary 4.4]{EKSV}
\label{thm:Fiedler} If 
$p(x)=x^n+c_{n-1}x^{n-1}+\cdots+c_{1}x+c_0$
is a polynomial over $\mathbb{R}$ with $n \geq 2$,
then $F$ is an $n \times n$ Fiedler companion matrix to 
$p(x)$ if and only if $F$ is equivalent to a unit lower 
Hessenberg matrix as in (\ref{mat:sparseForm}) with the
additional property that if $-c_{k}$ is in position 
$(i, j)$ then $-c_{k+1}$ is in position $(i-1, j)$ or 
$(i,j+1)$ for $1 \leq k \leq n-1$.
\end{theorem}

An alternative way to describe the unit lower
Hesenberg matrix in Theorem \ref{thm:Fiedler} is
to say that the variable entries of $R$
in \eqref{mat:sparseForm} form a lattice-path from
the bottom-left corner to the top-right corner of $R$.
The first two matrices in Figure \ref{fig:excom} are examples of Fiedler companion matrices since the variable entries of $R$ form a lattice-path. The last matrix in Figure \ref{fig:excom} is  not a Fiedler companion matrix.
 
 If $F$ is a Fiedler companion matrix, the \emph{initial step size} 
of $F$  is the number of coefficients other than $c_0$ 
in the row or column containing both $c_0$ and $c_{1}$.
The first matrix in Figure \ref{fig:excom} has initial step 
size three and the second matrix in Figure~\ref{fig:excom} has
initial step size one. 

\begin{remark}
Note that equivalent matrices have the same initial
step size since 
transpositions and permutation equivalence does not
change the number of coefficients in the row or 
column containing $c_0$ and $c_1$.
\end{remark}

Using Theorem \ref{thm:Fiedler} and Lemma \ref{lem:inverse},  one can describe the nonzero entries of the inverse of a Fiedler
companion matrix:

\begin{lemma}\cite[Theorem 3.2]{dTDP}
\label{thm:inventries}
Let $p(x) = x^n + c_{n-1}x^{n-1} + c_{n-2}x^{n-2} + 
\cdots + c_{1}x + c_0$ be a polynomial 
over $\mathbb{R}$ with $n\geq 2$ and $c_0 \neq 0$. Let $F$ 
be a Fiedler companion matrix to $p(x)$
with an initial step size $t$.
Then
\begin{enumerate} 
\item $F^{-1}$ has $t+1$ entries equal to $-\frac{1}{c_0},-\frac{c_{1}}{c_0},\ldots, -\frac{c_{t}}{c_0}$, 
\item $F^{-1}$ has $n-1-t$ entries equal to $c_{t+1},c_{t+2},\ldots,c_{n-1}$, 
\item $F^{-1}$ has $n-1$ entries equal to 1, and
\item the remaining entries of $F^{-1}$ are 0. 
\end{enumerate}
\end{lemma}

\begin{proof} 
Since $F$ is a companion matrix to $p(x)$, 
by Theorem \ref{thm:sparse}, the matrix $F$ is equivalent
to a lower Hessenberg matrix $C$ of the
form \eqref{mat:sparseForm}.  
Since $F$ and $C$ are equivalent, 
it follows that
the matrices $F^{-1}$ and $C^{-1}$ are 
equivalent, so it suffices to show that the matrix $C^{-1}$ satisfies conditions $(1)-(4)$.

Since $F$ is a Fielder companion matrix, Theorem
\ref{thm:Fiedler} implies that $c_1$ is either directly above $c_0$ in
$C$ or directly to the right of $c_1$.
If $c_1$ is to right of 
$c_0$ in $C$, then all other entries in the column containing $c_0$
is zero.  Alternatively, if $c_1$ is above $c_0$, all
entries to the right of $c_0$ in $C$ are zero.  

Lemma \ref{lem:inverse}, which gives us the inverse of a unit
lower Hessenberg matrix, applies to the matrix $C$.  By our above 
observation, the vector ${\bf u}$ or the vector ${\bf y}$ must be 
the zero vector. 
Without loss of generality, let ${\bf y}^T$ be zero, which means that 
$-\frac{1}{c_0}{\bf u}{\bf y}^T - H = -H$. If the initial step size
of $A$ is $t$, then there will be $t$ nonzero elements in ${\bf u}$, and it will have 
the form
\begin{equation*}
{\bf u} = \begin{bmatrix}
0 \\
\vdots \\
0 \\ 
-c_{t} \\
\vdots \\
-c_1 
\end{bmatrix} . 
\end{equation*}
By Lemma \ref{lem:inverse} the inverse of the matrix $C$ then has the form
\renewcommand{\arraystretch}{1.5}
\begin{equation}\label{matrixinverse} 
C^{-1} = \left[
\begin{array}{c|c|c}
	{\bf 0}^T& {\bf 0}^T & {-\frac{1}{c_0}} \\[6pt] \hline
	I_m & O& {\bf 0} \\[3pt] \hline 
    
	-H & I_{n-m-1} & \frac{1}{c_0}{\bf u}
\end{array}
\right].
\end{equation}
From \eqref{matrixinverse}, we can
describe the entries of $C^{-1}$:
$m+n-m-1 = n-1$ entries are $1$ (coming from
the submatrices $I_{m}$ and $I_{n-m-1}$); 
$c_{t+1}, \ldots, c_{n-1}$, which
all belong to the submatrix $-H$; the entry 
$-\frac{1}{c_0}$ from the top-right corner;
and the entries
$-\frac{c_1}{c_0},\ldots,-\frac{c_t}{c_0}$
from the term $\frac{1}{c_0}{\bf u}$.  Moreover,
the rest of the entries of $C^{-1}$ are zero.  
We have now shown that $C^{-1}$, and hence $F^{-1}$, has the desired
properties.
\end{proof}

\begin{remark}
Lemma ~\ref{thm:inventries} mimics \cite[Theorem 3.2]{dTDP}. As
observed in~\cite{VV},  the initial step size of a Fiedler companion 
matrix is equal to the number of initial 
consecutions or inversions of the permuation
associated with the Fielder companion matrix, as defined in \cite{dTDP}.
\end{remark}

We can now compute the condition number
for any Fiedler companion matrix.  This
result first appeared in \cite{dTDP}, 
but we can avoid 
the  formal analysis of the permutation that was used to construct
the Fiedler companion matrix, as well as the associated concepts of consecution
and inversion of a permutation.

\begin{theorem}\cite[Theorem 4.1]{dTDP}
\label{thm:FCondNumber}
Let $p(x) = x^n + c_{n-1}x^{n-1} + c_{n-2}x^{n-2} + \cdots + c_{1}x + c_0$ be a polynomial over $\mathbb{R}$ with $n\geq 2$ and $c_0 \neq 0$. Let $F$ be a Fiedler companion matrix to $p(x)$ with an initial step size $t$. Then
\begin{equation}\nonumber
    \kappa(F)^2 = ||F||^2\cdot\left((n-1) + \frac{1 + |c_{1}|^2+\cdots+|c_{t}|^2}{|c_0|^2}+|c_{t+1}|^2+ \cdots +|c_{n-1}|^2\right),
\end{equation}
with
\begin{equation}\nonumber
    ||F||^2 = (n-1)+|c_{0}|^2 + |c_{1}|^2 + \cdots + |c_{n-1}|^2.
\end{equation}
\end{theorem}

\begin{proof}
This result follows from the fact
that $F$ is a unit sparse companion matrix (so it contains
$n-1$ entries equal to $1$ and the entries $-c_0,\ldots,-c_{n-1}$), and Lemma \ref{thm:inventries}, which
describes the entries of $F^{-1}$.
\end{proof}

Because the condition number
$\kappa(F)$ of a Fiedler companion matrix 
$F$ depends only upon the 
initial step size and not the permutation
$\sigma$, we can derive the following corollary.

\begin{corollary}\label{Cor:step}
\textup{\cite[Corollary~4.3]{dTDP}} Let $
p(x) = x^n + c_{n-1}x^{n-1} + c_{n-2}x^{n-2} + \cdots + c_{1}x + c_0$ be a polynomial over $\mathbb{R}$ with $n\geq 2$ and $c_0 \neq 0$. Let $A$ and $B$ be Fiedler companion matrices to the polynomial $p(x)$. 
If the initial step size of both $A$ and $B$ is $t$, 
then $\kappa(A) = \kappa(B)$.
\end{corollary}

Since condition numbers of Fiedler companion matrices
depend on the
initial step size, let
$$S_t = \{F ~|~ \mbox{$F$ is a Fiedler companion matrix to $p(x)$ with initial step size $t$}\},$$ 
and define $
    \kappa(t) = \kappa(F) ~~\mbox{for $F \in S_t$}.$
We can now recover a result of \cite{dTDP}
that allows us to compare the condition numbers
of Fielder matrices while again avoiding any reference to the permutation $\sigma$
used to define a Fiedler matrix.

\begin{corollary}
\cite[Corollary 4.5]{dTDP}
\label{Dopico45}
Let $p(x) = x^n + c_{n-1}x^{n-1} + c_{n-2}x^{n-2} + \cdots + c_1x + c_0$ be a polynomial over $\mathbb{R}$ with $n\geq 2$ and $c_0 \neq 0$.  
Then
\begin{enumerate}
    \item if $|c_0|<1$, then $\kappa(1) \leq \kappa(2) \leq \cdots \leq \kappa(n-1)$;
    \item if $|c_0|=1$, then $\kappa(1) = \kappa(2) = \cdots = \kappa(n-1)$; and
    \item if $|c_0|>1$, then $\kappa(1) \geq \kappa(2) \geq \cdots \geq \kappa(n-1)$.
\end{enumerate}
\end{corollary}

\begin{proof}
Note that by Corollary \ref{Cor:step}, $\kappa(A)$
is the same for all $A \in S_t$, so $\kappa(t)$ is 
well-defined.  The conclusions follow from
Theorem \ref{thm:FCondNumber}.
\end{proof}

One of our new results is to compare the 
condition number of a Fielder companion matrix
of $p(x)$ to the condition number of other companion matrices of $p(x)$.
In particular, if a Fiedler companion
matrix $F$ has a smaller condition number than another
companion matrix $C$ to the same polynomial $p(x)$, then  the ratio $\frac{\kappa(C)}{\kappa(F)}$
can be bounded.
This result is similar in spirit to
\cite[Theorem 4.12]{dTDP}.

\begin{theorem}
\label{HeftyThm}
Let $p(x) = x^n + c_{n-1}x^{n-1} + \cdots + c_1x + c_0$ be a polynomial over $\mathbb{R}$ with $n\geq 2$, and $c_0 \neq 0$.  Let $F$ be a Fielder companion
matrix to $p(x)$.
Further, suppose $C$ is any companion matrix to $p(x)$ 
whose lower Hessenberg form is
\renewcommand{\arraystretch}{1.5}
\begin{equation*}
C = \left[
\begin{array}{c|c|c}
	{\bf 0} & I_m & O \\[6pt] \hline
	{\bf u}_C & H_C & I_{n-m-1} \\[3pt] \hline 
	{-c_0} & {\bf y}^T_C & {\bf 0}^T
\end{array}
\right]
\end{equation*}
such that either ${\bf u}_C$ or ${\bf y}^T_C$ is the zero vector. 
If $\kappa(F) \leq \kappa(C)$, then
\begin{equation*}
1 \leq \frac{\kappa(C)}{\kappa(F)} \leq \kappa(F) .
\end{equation*}
\end{theorem}

\begin{proof}
The conclusion that $1 \leq \frac{\kappa(C)}{\kappa(F)}$ is immediate from the 
hypothesis that $\kappa(F) \leq \kappa(C)$.     

By  Theorem~\ref{thm:Fiedler}
and Lemma~\ref{equivalentcondition}, we can assume $F$ is in unit lower Hessenberg form. As such, 
let \renewcommand{\arraystretch}{1.5}
\begin{equation*}
F = \left[
\begin{array}{c|c|c}
	{\bf 0} & I_l & O\\[6pt] \hline
	{\bf u}_F & H_F & I_{n-l-1} \\[3pt] 
\hline 
	{-c_0} & {\bf y}^T_F & {\bf 0}^T
\end{array}
\right].
\end{equation*}
and let $t$ be the initial step size of $F$.  
We want to show that
\begin{equation*}
    \frac{||C||\cdot||C^{-1}||}{||F||\cdot||F^{-1}||} \leq ||F||\cdot||F^{-1}||.
\end{equation*}
Since $C$ and $F$ are unit sparse companion matrices, $||C||=||F||.$
It suffices to show that
\begin{equation*}
\label{Result}
||C^{-1}|| \leq ||F||\cdot||F^{-1}||^2.
\end{equation*}

Using equivalence, we may assume without loss of generality that ${\bf u}_C = {\bf 0}$. 
By Lemma~\ref{lem:inverse},
\renewcommand{\arraystretch}{1.5}
\begin{equation*}
C^{-1} = \left[
\begin{array}{c|c|c}
	\frac{1}{c_0}{\bf y}^T_C & {\bf 0}^T & {-\frac{1}{c_0}} \\[6pt] \hline
	I_m & O& {\bf 0} \\[3pt] \hline 
    - H_C & I_{n-m-1} & 
    {\bf 0} 
\end{array}
\right].
\end{equation*}
since ${\bf u}_C=\bf{0}$.
Then
\begin{equation}\label{CSummand}
    ||C^{-1}||^2 
   =  (n-1) + \left(\frac{1}{c_0}\right)^2 + \sum_{c_i \in {\bf y}^T_C} \left|\frac{c_i}{c_0}\right|^2 +  \sum_{c_k \in H_C} |c_{k}|^2.
\end{equation}
where $c\in H$ (resp. $c\in \mathbf{y}$) means $-c$ is an entry in $H$ (resp. $\mathbf{y}$).
On the other hand, using Lemma \ref{thm:inventries},
\begin{equation}
||F||^2\cdot ||F^{-1}||^4 =
\left[(n-1) + \sum_{i=0}^{n-1} |c_i|^2\right] 
\left[(n-1) + \left(\frac{1}{c_0}\right)^2 + \sum_{i=1}^{t}\left|\frac{c_i}{c_0}\right|^2 + \sum_{j=t+1}^{n-1}|{c_j}|^2 \right]^2.\label{MSummand}
\end{equation}
We want to show that $||C^{-1}|| \leq ||F||\cdot||F^{-1}||^2$ which is equivalent
to showing that $||C^{-1}||^2 \leq 
||F||^2\cdot ||F^{-1}||^4$.  To do this, 
for each of the four different summands in \eqref{CSummand}, we show that there exists distinct terms in $||F||^2\cdot||F^{-1}||^4$ that are greater than or equal to the summand.  Here we rely on the
fact that there
are no negative summands in \eqref{MSummand}.

Partially expanding out \eqref{MSummand}, we have
\begin{multline*}
||F||^2 \cdot ||F^{-1}||^4  = 
(n-1)^3 + (n-1)^2\left(\frac{1}{c_0}\right)^2 +
(n-1)\left(\sum_{i=0}^{n-1} |c_i|^2\right)\left(\frac{1}{c_0}\right)^2\\
+ (n-1)^2\sum_{j=0}^{n-1}|{c_j}|^2 
+ \mbox{other non-negative terms.}
\end{multline*}
Consequently,
\small
\begin{eqnarray*}
||C^{-1}||^2 & = & 
(n-1) + \left(\frac{1}{c_0}\right)^2 + \sum_{c_i \in {\bf y}^T_C} \left|\frac{c_i}{c_0}\right|^2 +  \sum_{c_k \in H_C} |c_{k}|^2 \\
& \leq & (n-1)^3 + (n-1)^2\left(\frac{1}{c_0}\right)^2 +
(n-1)\left(\sum_{i=0}^{n-1} |c_i|^2\right)\left(\frac{1}{c_0}\right)^2+ (n-1)^2\sum_{j=0}^{n-1}|{c_j}|^2\\
&\leq &||F||^2 \cdot ||F^{-1}||^4. 
\end{eqnarray*}
\normalsize
\end{proof}


\section{Striped Companion Matrices}\label{sec:stripe}

 In this section we explore a particular class
 of companion matrices
  known as striped companion matrices, which were introduced in \cite{EKSV}.  A striped companion matrix to a polynomial $p(x) = x^n + c_{n-1}x^{n-1} + \cdots + c_1x + c_0$ has the property that the coefficients $-c_0, -c_1, \ldots, -c_{n-1}$ form horizontal stripes in the matrix.
In particular, if $\mathbf{t}=(t_1,t_2,\ldots,t_r)$ is an ordered $r$-tuple of positive integers with 
$t_1+t_2+\cdots+t_r=n$, and $t_1\geq t_i$ for $2\leq i\leq n$, then we define the \emph{striped companion matrix} $S_n(\mathbf{t})$
to be the companion matrix of  unit Hessenberg form 
\begin{eqnarray}\label{Rmatrix}
S_n(\mathbf{t})=\left[\begin{array}{c|c}
\mathbf{0}\quad I_{t_1-1}&O \\  \hline 
\multirow{2}{*}{$R$}  & I_{n-t_1}\\
 & \mathbf{0}^T
\end{array}\right]
\end{eqnarray}
with the $(n-t_1+1)\times t_1$ matrix $R$  having $r$ nonzero
rows and with the $i$\textsuperscript{th} nonzero row of $R$ having $t_i$ variables in the
first $t_i$ positions and $t_i-1$ zero rows immediately above it in $R$, for~$1< i\leq r$.
Note that this implies the first row of $R$ is a nonzero row with $t_1$ leading nonzero entries. For example,
\footnotesize
\[ 
S_7(3,2,2) = 
\begin{small}\left[\begin{array}{rrrrrrr}
0 & 1 & 0 & 0 & 0 & 0 & 0 \\
0 & 0 & 1 & 0 & 0 & 0 & 0 \\
-c_4 & -c_5 & -c_6 & 1 & 0 & 0 & 0\\
0 & 0 & 0 & 0 & 1 & 0 & 0\\
-c_2 & -c_3 & 0 & 0 & 0 & 1 & 0\\
0 & 0 & 0 & 0 & 0 & 0 & 1 \\
-c_0 & -c_1 & 0 & 0 & 0 & 0 & 0
\end{array}\right]
\end{small}, ~~\mbox{and}~~
S_8(3,3,2)=
\begin{small}
\left[\begin{array}{rrrrrrrr}
0 & 1 & 0 & 0 & 0 & 0&0&0 \\
0 & 0 & 1 & 0 & 0 & 0&0&0 \\
-c_5 & -c_6 & -c_7 & 1 & 0 & 0&0&0 \\
0 & 0 & 0 & 0 & 1 & 0&0&0 \\
0 & 0 & 0 & 0 & 0 & 1 &0&0\\
-c_2 & -c_3 & -c_4 & 0 & 0 & 0 &1&0\\ 
0 & 0 & 0 & 0 & 0 & 0 &0&1\\
-c_0 & -c_1 & 0 & 0 & 0 & 0 &0&0
\end{array}\right]\end{small}.
\]
\normalsize

As the next theorem shows,
in some cases the 
stripped companion matrices can have a better
condition number than a
Fielder companion matrix.

\begin{theorem}
\label{RectangleInequality}
Suppose $n=k(m+1)$ for some positive $k,m\in \mathbb{Z}$ and 
$p(x) = x^n + c_{n-1}x^{n-1} 
+ \cdots + c_1x + c_0$
with $c_0 = 1$, $c_1,\ldots,c_{n-1} \in \mathbb{R}$. 
There exists a striped companion matrix $S=S_n(k,\ldots,k)$  
for $p(x)$ such that $\kappa(S) \leq \kappa(F)$ for every Fiedler companion matrix $F$ if and only if 
\begin{equation}\label{eq:stripe}
\sum_{j=1}^m\left(\sum_{i=1}^{k-1} |c_ic_{jk} - c_{jk + i}|^2\right) \leq \sum_{j=1}^m\left(\sum_{i=1}^{k-1}|c_{jk + i}|^2\right) .
\end{equation}
\end{theorem}

\begin{proof}
Let $S=S_{k(m+1)}(k,\ldots,k)$, and let $F$ be a Fiedler companion matrix. 
Since $||S||=||F||$ as noted in Remark~\ref{rem:same},  it suffices to show that $||S^{-1}|| \leq ||F^{-1}||$ if and only if 
equation (\ref{eq:stripe}) holds.
By Lemma~\ref{lem:inverse}, 
\begin{equation*} 
\label{InverseForU}
S^{-1} = \footnotesize \left[
\begin{array}{c|c|c}
	\begin{matrix} -c_{1}\;\;\;\;\;\;\;\;\;\;\;\; & \;\;\;\;\;\;-c_{2}\;\;\;\;\;\;\;\;\; & \;\;\;\ldots\;\;\;\;\;\; & \;\;\;\;\;\;\;\;\;-c_{k-1}\end{matrix} & \mathbf{0}^T & -1 \\[6pt] \hline
	I_{k-1} & O & \mathbf{0} \\[3pt] \hline 
	\begin{matrix}
-c_1c_{mk} + c_{mk+1} & -c_2c_{mk} + c_{mk+2} & \ldots & -c_{k-1}c_{mk} + c_{(m+1)k-1}  \\
0 & 0 & \ldots & 0  \\
\vdots & \vdots & \ldots & \vdots  \\
\vdots & \vdots & \ldots & \vdots  \\
0 & 0 & \ldots & 0  \\
-c_1c_{2k} + c_{2k+1} & -c_2c_{2k} + c_{2k+2} & \ldots & -c_{k-1}c_{2k} + c_{3k-1} \\
0 & 0 & \ldots & 0  \\
\vdots & \vdots & \ldots & \vdots  \\
0 & 0 & \ldots & 0  \\
-c_1c_{k} + c_{k+1} & -c_2c_{k} + c_{k+2} & \ldots & -c_{k-1}c_{k} + c_{2k-1}  \\
0 & 0 & \ldots & 0  \\
\vdots & \vdots & \ldots & \vdots  \\
0 & 0 & \ldots & 0 
\end{matrix} & I_{mk} & \begin{matrix}
-c_{mk} \\
0 \\
\vdots \\ \vdots \\ 
0 \\
-c_{2k} \\
0 \\
\vdots \\ 
0 \\
-c_{k} \\
0 \\
\vdots \\ 
0 
\end{matrix}
\end{array}
\right]. 
\end{equation*}
\normalsize
Thus 
$$
||S^{-1}||^2 = n + 
\sum_{j=1}^{k-1} |c_j|^2
+\sum_{j=1}^{m} |c_{jk}|^2
+\sum_{j=1}^m\left(\sum_{i=1}^{k-1} |c_ic_{jk} - c_{jk + i}|^2\right). 
$$
By Theorem \ref{thm:FCondNumber}, 
$$
||F^{-1}||^2 = n +
\sum_{j=1}^{k-1} |c_j|^2
+\sum_{j=1}^{m} |c_{jk}|^2
+\sum_{j=1}^m\left(\sum_{i=1}^{k-1}|c_{jk + i}|^2\right). 
$$
Therefore 
$\kappa(S)\leq \kappa(F)$  if and only if
\begin{equation*}
\sum_{j=1}^m\left(\sum_{i=1}^{k-1} |c_ic_{jk} - c_{jk + i}|^2\right) \leq \sum_{j=1}^m\left(\sum_{i=1}^{k-1}|c_{jk + i}|^2\right) .
\end{equation*}
\end{proof}

We can deduce the following
corollary.

\begin{corollary} 
\label{CoeffCor}
Suppose $n = k(m+1)$ for some $m,k\in\mathbb{Z}$ and
$ 
p(x) = x^n + c_{n-1}x^{n-1} 
+ \cdots + c_1x + c_0
$ 
with $c_0 = 1$, $c_1,\ldots,c_{n-1} \in \mathbb{R}$.  Suppose $F$ is any Fiedler
companion matrix for $p(x)$. If 
\begin{equation*}
|c_ic_{jk} - c_{jk+i}| \leq |c_{jk+i}|,
~~\mbox{for $1\leq j\leq m$ and $1 \leq i \leq k-1$},
\end{equation*}
then there exists a 
striped companion matrix $S=S_n(k,\ldots,k)$, such that $\kappa(S) \leq \kappa(F)$. 
\end{corollary}

\begin{example} Let 
\begin{equation*}
p(x) = x^{9} + 8x^{8} + 6x^{7} + 2x^{6} + 5x^{5} + 8x^{4} + 3x^{3} + 3x^{2} + 2x + 1.
\end{equation*}
Note that the inequalities in Corollary~\ref{CoeffCor} hold.
Let $F$ be any Fiedler companion to $p(x)$ 
and consider the striped companion matrix $S=S_9(3,3,3)$, i.e.,
\[ S_9(3,3,3)=
\begin{small}
\left[\begin{array}{rrrrrrrrr}
0 & 1 & 0 & 0 & 0 & 0&0&0&0 \\
0 & 0 & 1 & 0 & 0 & 0&0&0&0 \\
-2 & -6 & -8 & 1 & 0 & 0&0&0 &0\\
0 & 0 & 0 & 0 & 1 & 0&0&0&0 \\
0 & 0 & 0 & 0 & 0 & 1 &0&0&0\\
-3 & -8 & -5 & 0 & 0 & 0 &1&0&0\\ 
0 & 0 & 0 & 0 & 0 & 0 &0&1&0\\
0 & 0 & 0 & 0 & 0 & 0 &0&0&1\\
-1 & -2 & -3& 0 & 0 & 0 &0&0&0
\end{array}\right]\end{small}.
\]
Then $||S||=||F||=\sqrt{224}$,
but $\kappa(S)=\sqrt{224}\sqrt{63}<\kappa(F) =
\sqrt{224}\sqrt{224}.$ 
\end{example}

One extreme example of how the inequalities in Corollary~\ref{CoeffCor} can be met is if $c_0 =1$ and
the striped companion matrix in line (\ref{Rmatrix}) 
has rank$(R)=1$. In this case, the inequalities
are trivially met as described in the following corollary. A more general result can be developed for striped companion matrices with differing stripe lengths; e.g., see \cite[Section 4.2]{C}.

\begin{corollary}
\label{BThm}
Given $p(x) = x^n + c_{n-1}x^{n-1} + c_{n-2}x^{n-2} + \cdots + c_1x + c_0$
with $c_0 = 1$, and $c_1, \ldots, c_{n-1} \in \mathbb{R}$, let $S$ be a striped companion matrix to the polynomial $p(x)$. If
\begin{equation*}
S = \left[ \begin{array}{c|c}
	\mathbf{0} \;\; I_m & O \\[6pt] \hline 
	& \\[-10pt]
	R & \begin{matrix} 
    I_{n-m-1} \\ 
    \mathbf{0}^T 
    \end{matrix}
\end{array} \right]
\end{equation*}
with 
rank$(R)=1$,
then $\kappa(S) \leq \kappa(F)$ for any Fiedler companion matrix $F$.
\end{corollary}

\begin{proof}
This result follows from Corollary~\ref{CoeffCor} by observing 
that $|c_ic_{jk} - c_{jk+i}|=0$ 
for all $1\leq j\leq m$ and $1 \leq i \leq k-1$,
if and only if rank$(R)=1$.
In particular, rank$(R)=1$ if and only 
every $2\times 2$ submatrix of $R$ has determinant zero, 
which is true
if and only if $|c_ic_{jk} - c_{jk+i}|=0$ 
for $1\leq j\leq m$ and $1 \leq i \leq k-1$.
Note that we are using the fact that
$$\begin{bmatrix} -c_{jk} & -c_{jk+i} \\ 
-c_0 & -c_i \end{bmatrix}$$
is a $2 \times 2$ submatrix of $R$ and 
$c_0 = 1$.
\end{proof}

\begin{example}
\label{strippedcompanion}
Let $b,k\in \mathbb{R}$ and consider the polynomial
$p(x)=x^6+(bk^3)x^5+(bk^2)x^4+(bk^2)x^3+(bk)x^2+
kx+1$. 
If 
\[S=S_6(2,2,2)= 
\left[\begin{array}{rrrrrr}
0 & 1 & 0 & 0 & 0 & 0 \\
-bk^2 & -bk^3 & 1 & 0 & 0 & 0 \\
0 & 0 & 0 & 1 & 0 & 0 \\
-bk & -bk^2 & 0 & 0 & 1 & 0 \\
0 & 0 & 0 & 0 & 0 & 1 \\
-1 & -k & 0 & 0 & 0 & 0 
\end{array}\right]\]
and $F$ is
any Fiedler companion matrix for $p(x)$, then 
\[\left(\frac{\kappa(F)}{\kappa(S)}\right)^2=\frac{b^2k^6 + b^2k^4 + b^2k^4 + b^2k^2  + k^2 + 6}{b^2k^4 + b^2k^2 +k^2 +6}
.\]
In this case, for sufficiently large $k$, 
\[\frac{\kappa(F)}{\kappa(S)}\approx k\]
demonstrating a significantly better condition number for $S$ compared to any Fiedler companion matrix.
\end{example}

As shown in Corollary \ref{BThm}, if the rank
of the submatrix $R$ in the striped companion
matrix $S$ has ${\rm rank}(R) = 1$, then the 
inequality $\kappa(S) \leq \kappa(F)$
holds for any Fiedler companion matrix $F$.
Note that in the striped companion matrix given in
Example \ref{strippedcompanion}, the corresponding 
submatrix $R$ has rank one.  Observe also that
we can write $p(x)$ has 
$$p(x) = q(x)+ (bk)x^2q(x) + (bk^2)x^4q(x) + x^6
~~\mbox{with $q(x) = 1+kx$}.$$

\noindent
This generalizes: if the matrix $S$ in Corollary~\ref{BThm} has ${\rm rank}(R)=1$, then $p(x)= x^n+q(x)f(x)$ for some polynomial $q(x)$ with $\deg(q(x))=m$ and $\deg(f(x))=n-m-1.$ 
Moreover, Corollary~\ref{BThm} can be improved by giving an estimate on
$\frac{\kappa(F)}{\kappa(S)}$ for any
Fiedler companion matrix $F$.

\begin{theorem}
Suppose $n=k(m+1)$ and
$p(x) = q(x) + b_1x^kq(x) + b_2x^{2k}q(x) + \cdots + b_mx^{mk}q(x) + x^{(m+1)k}$ with
\begin{equation*}
    q(x) = a_{k-1}x^{k-1} + a_{k-2}x^{k-2} + \cdots + a_1x + 1.
\end{equation*}
Let $S = S_n(k,k,\ldots,k)$ and $F$ be any
Fiedler companion matrix to $p(x)$.
If $(b_1^2 + \cdots + b_{m}^2)$ is sufficiently large, then 
\begin{equation*}
\left(\frac{\kappa(F)}{\kappa(S)}\right)^2 
    \approx (a_1^2 + \cdots + a_{k-1}^2+1),
\end{equation*}
or if $(a_1^2 + \cdots + a_{k-1}^2)$ is sufficiently large, then 
\begin{equation*}
\left(\frac{\kappa(F)}{\kappa(S)}\right)^2 
    \approx (1+b_1^2 + \cdots + b_{m}^2).
\end{equation*}
\end{theorem}

\begin{proof}
By Remark~\ref{rem:same},
$\frac{\kappa(F)}{\kappa(S)} =\frac{||F^{-1}||}{||S^{-1}||}$.
By Lemma~\ref{lem:inverse}, $$||S^{-1}||^2=a_1^2+\cdots +a_{k-1}^2+b_1^2+\cdots+b_m^2+n.$$
By Theorem~\ref{thm:FCondNumber} we can determine that
$$||F^{-1}||^2=(1+b_1^2+\cdots+b_{m}^2)(a_1^2+\cdots+a_{k-1}^2)+(b_1^2+\cdots+b_{m}^2)+n.$$
Therefore,
\begin{equation*}
\left[\frac{\kappa(F)}{\kappa(S)}\right]^2 = 
\frac{(1+b_1^2 + \cdots + b_{m}^2)(a_1^2 + \cdots + a_{k-1}^2)+(b_1^2 + \cdots + b_{m}^2) +n}{(a_1^2 + \cdots +a_{k-1}^2) + (b_1^2 + \cdots + b_{m}^2) + n}
\end{equation*}
and the result follows.
\end{proof}


\section{Generalized 
companion matrices: a case study}\label{sec:gen}

In the previous sections, we
focused on the condition numbers
of unit sparse companion matrices.  In
this section, we initiate an
investigation into the condition numbers of a 
family of matrices that are not
companion matrices, but have properties 
similar to companion matrices. To
date, there appears to be little work done 
on this approach, so the 
work in this section can be seen as 
providing a proof-of-concept for future projects.
These results can also be viewed in the broader context
of developing
the properties of generalized companion
matrices (e.g., see \cite{EKSV,GSSV}).
Roughly speaking, 
given a polynomial $p(x) = x^n + c_{n-1}x^{n-1}
+ \cdots + c_1x_1 + c_0$, 
a generalized companion matrix $A$ is a matrix
whose entries are polynomials 
in the $c_0,\ldots,c_n$ and whose 
characteristic polynomial is $p(x)$.
See \cite{GSSV} for more explicit detail. 

Instead of considering the general
case, we focus on a particular family of matrices
and their condition numbers. This case study shows that the condition numbers
can improve on 
those of
Frobenius (or Fiedler) companion matrices under some
extra hypotheses.  

We now define our special family.  Let 
$p(x) = x^n + c_{n-1}x^{n-1} + \cdots + c_1x + c_0$
be a polynomial over $\mathbb{R}$ with $n \geq 2$
and let $a \in \mathbb{R}$ be any real number.   
Fix an integer 
$\ell \in \{3,\ldots,n-2\}$ and let
\begin{eqnarray*}
{\bf a}^T &=& (-c_{n-1},-c_{n-2},\ldots,-c_{\ell +1}) ~~\mbox{and} 
~~{\bf b}^T  =(-c_{\ell-2},-c_{\ell-3},\ldots,-c_1).
\end{eqnarray*}
Then let 
\renewcommand{\arraystretch}{1.5}
\begin{equation}\label{eq:M}
M_n(a,\ell)= \left[
\begin{array}{c|c|c|c}
	  {\bf a} & I_{n-\ell-1} & O & O\\[3pt] \hline 
	-c_{\ell}+a & \multirow{2}{*}{$W$}  & \multirow{2}{*}{$I_2$} & \multirow{2}{*}{$O$} \\[3pt]
    -c_{\ell-1}+ac_{n-1} &  &  & \\[3pt] 
    \hline
     {\bf b} & O & O & I_{\ell-2} \\[3pt] 
    \hline
    -c_0 & O & O & O \\
\end{array}
\right].
\end{equation}
where $W$ is a $2 \times (n-\ell-1)$ matrix having $W_{2,1} =-a$
and zeroes in every other entry.  

Informally, the matrix $M_n(a,\ell)$
is constructed by starting with the Frobenius companion matrix
which has all the coefficents of $p(x)$ in the first
column.  Then we fix a row  that is neither the top row
nor one of the bottom two rows (this corresponds
to picking the $\ell$), and then adding $a$ to $c_{\ell}$
in the $(n-\ell)$-th row, and $-a$ in the column to the right
and one below.  We then also add $ac_{n-1}$ to the
first  entry in the $(n-\ell+1)$-th row.  Note that when
$a =0$, $M_n(0,\ell)$ is equivalent to
the Frobenius companion matrix.  We can
thus view $M_n(a,\ell)$ as a perturbation of the
Frobenius companion matrix when $a\neq 0$.  As an example,
the matrix $M_7(a,4)$ is given in 
Figure \ref{fig:matrixM7(a,4)}.
\begin{figure}
     $$\left[\begin{array}{rrrrrrr}
     -c_6  &1&0&0&0&0&0\\
     -c_5  &0&1&0&0&0&0\\
     -c_4+a  &0&0&1&0&0&0\\
     -c_3+ac_6  &-a&0&0&1&0&0\\
     -c_2  &0&0&0&0&1&0\\
     -c_1  &0&0&0&0&0&1\\
     -c_0  &0&0&0&0&0&0\\
     \end{array}
     \right]
     $$
    \caption{The matrix $M_7(a,4)$}
    \label{fig:matrixM7(a,4)}
\end{figure}

We wish to compare the condition number of
$M_n(a,\ell)$ with the Frobenius (and Fiedler) companion
matrices.  In some cases our new
matrix $M_n(a,\ell)$ can provide us with a smaller 
condition number.  The next lemma gives the inverse of 
$M_n(a,\ell)$ and shows that the characteristic polynomial of
$M_n(a,\ell)$ is $p(x)$.

\begin{lemma}\label{propertiesofM}
Let $p(x) = x^n + c_{n-1}x^{n-1} + c_{n-2}x^{n-2} + \cdots + c_{1}x + c_0$ be a polynomial over $\mathbb{R}$, with $n\geq 2$ and $c_0 \neq 0$.  Let $a \in \mathbb{R}$ and 
$\ell \in \{3,\ldots,n-2\}$, and let $M=M_n(a,\ell)$ be constructed
from $p(x)$ as above.  Then
\begin{enumerate}
\item[$(i)$] the characteristic polynomial of $M$
is $p(x)$, and
\item[$(ii)$] if $c_0 \neq 0$, then 
\renewcommand{\arraystretch}{1.5}
$$
M^{-1}= \frac{1}{c_0}\left[
\begin{array}{c|c|c|c}
	   {\bf 0}^T & {\bf 0}^T & {\bf 0}^T & -1 \\ [3pt] \hline
      c_0I_{n-\ell} & O & O & {\bf a} \\ [3pt] \hline
      \multirow{2}{*}{$-c_0W$} & \multirow{2}{*}{$c_0I_2$} & \multirow{2}{*}{$O$} & -c_{\ell} + a \\
            &        &   &  -c_{\ell-1} \\ [3pt] \hline
        O & O & c_0I_{\ell-2} & {\bf b} \\
\end{array}
\right].
$$
\end{enumerate}
\end{lemma}

\begin{proof}
$(i)$ 
We employ the fact that the determinant of a matrix is a linear function of its rows. 
In particular, if 
$M=M_n(a,\ell)$, we observe that row $n-\ell$ of 
$xI_n-M$ can be written as $\mathbf{u}+a\mathbf{v}$ for some vectors $\mathbf{u}$ and $\mathbf{v}$ such that $\mathbf{u}$ is not a function of $a$. Row $n-\ell+1$ of $xI_n-M$ can also be written in a similar manner. Let $k=n-\ell.$
Thus
applying linearity to row $k$ gives us
\begin{eqnarray}\label{klinearstep}
&\det(xI_n-M)=\det\left(xI_n-\left[
\begin{tiny}\begin{array}{c|c|c|c}
	  {\bf a} & I_{n-\ell-1} & O & O\\[3pt] \hline 
	-c_{\ell} &\multirow{2}{*}{W}  &\multirow{2}{*}{$I_2$}  &\multirow{2}{*}{$O$} \\[3pt]
    -c_{\ell-1}+ac_{n-1} &  &  & \\[3pt] 
    \hline
     {\bf b} & O & O & I_{\ell-2} \\[3pt] 
    \hline
    -c_0 & O & O & O \\
\end{array}\end{tiny}
\right]\right) + a(-1)x^{\ell}. \end{eqnarray}
Note that the term $a(-1)x^{\ell}$ in 
\eqref{klinearstep} comes from
computing the determinant of
the matrix $A'$ formed by replacing the $k$-th row
of the matrix $xI_n -M$ with the row
$\begin{bmatrix} -a & 0 & \cdots
&0\end{bmatrix}$.  Doing a row expansion along
the $k$-th row of $A'$, the determinant of $A'$ is
$(-1)^{k+1}(-a){\rm det}(A'')$ where $A''$ is 
a block lower diagonal matrix with diagonal blocks
$D_1$ and $D_2$.  Furthermore, 
$D_1$ is a $(k-1) \times (k-1)$ lower triangular 
matrix with $-1$ on all the diagonal entries, and
$D_2$ is a $\ell \times \ell$ upper 
triangular matrix with $x$ on all the
diagonal entries.  So ${\rm det}(A'') = 
(-1)^{k-1}x^{\ell}$, and hence  ${\rm det}(A') =
(-1)^{k+1}(-a) (-1)^{k-1}x^{\ell}= 
(-a)x^{\ell}$, as desired.

We now apply 
linearity to row
$k+1$ in the matrix
that appears on the 
right-hand side of 
\eqref{klinearstep};  in particular, a similar argument shows that the right-hand side \eqref{klinearstep} is
equal to 
\begin{multline}\label{k+1linear}
\det\left(xI-\left[\begin{tiny}
\begin{array}{c|c|c|c}%
	  {\bf a} & I_{k-1} & O & O\\[3pt] \hline 
	-c_{\ell} & \multirow{2}{*}{$O$}  & \multirow{2}{*}{$I_2$}  & \multirow{2}{*}{$O$} \\[3pt]
    -c_{\ell-1} &  &  & \\[3pt] 
    \hline
     {\bf b} & O & O & I_{\ell-2} \\[3pt] 
    \hline
    -c_0 & O & O & O \\
\end{array}\end{tiny}
\right]\right) + a(-1)x^{\ell} +ac_{n-1}(-1)x^{\ell-1}+ a(x+c_{n-1})x^{\ell-1}.
\end{multline}
Note that
the first summand in 
\eqref{k+1linear} is the characteristic
polynomial of a Frobenius companion matrix of $p(x)$, and hence is  $p(x)$.  
Thus, \eqref{k+1linear} reduces to
$$p(x) + a(-1)x^{\ell} +ac_{n-1}(-1)x^{\ell-1}+ a(x+c_{n-1})x^{\ell-1} =p(x).$$
$(ii)$ A direct multiplication will show that the given matrix
is the inverse $M$.
\end{proof}

Because both $M_n(a,\ell)$ and its inverse are known,
we are able to compute its condition number.
In the next lemma, instead of providing the 
general formula, we compute the condition
number under the extra assumption that $c_0 =1$ in the
polynomial $p(x)$.

\begin{lemma}\label{specialcondnumb}
Let $p(x) = x^n + c_{n-1}x^{n-1} + c_{n-2}x^{n-2} + \cdots + c_{1}x + c_0$ be a polynomial over $\mathbb{R}$, with $n\geq 2$, and suppose that $c_0 = 1$.  Let $a \in \mathbb{R}$ and 
$\ell \in \{3,\ldots,n-2\}$, and let $M = M_n(a,\ell)$.
Then
$$\kappa(M)^2
= \left(v +a^2 + (a-c_{\ell})^2+(ac_{n-1}-c_{\ell-1})^2\right)
\left(v+a^2 + (a-c_{\ell})^2+c_{\ell-1}^2+1\right)$$
with $$v = n - c_{\ell-1}^2-c_{\ell}^2 + \sum_{i=1}^{n-1} c_i^2.$$
\end{lemma}

The next result illustrates the desired proof-of-concept.
In particular, the result shows that in special cases, the condition number of the matrix
$M_n(a,\ell)$, which has properties similar to
a companion matrix, has a condition
number smaller than \emph{any} Fielder companion matrix.   
Although the scope of this result is limited, it does
suggest that generalized companion matrices, and
in particular perturbations of the Frobenius companion 
matrix, can provide
better condition numbers in some cases.

\begin{theorem}
Let $n \geq 2$, and fix  $\ell \in \{3,\ldots,n-2\}$
and  $t \in \mathbb{R}$.  Set
$$p(x) = x^n + tx^{n-1}+tx^{\ell} + t^2x^{\ell-1}+1.$$
Let $M = M_n(t,\ell)$. 
Then, for
any Fieldler companion matrix $F$ of $p(x)$,
$$\frac{\kappa(F)^2}{\kappa(M)^2} = 
\frac{(n+2t^2+t^4)^2}{(n+2t^2)(n+1+2t^2+t^4)}.$$
In particular, for $t$ for sufficiently 
large, $\frac{\kappa(F)}{\kappa(M)} \approx
\frac{1}{\sqrt{2}}t$.
\end{theorem}

\begin{proof}
By Lemma \ref{specialcondnumb}, 
$$ 
\kappa(M)^2 = \left(1+t^2+(n-1)+a^2+(a-t)^2+(at-t^2)^2\right)
\left(1+t^2+(n-1)+a^2+(a-t)^2+t^4+1\right).$$
Setting $a=t$ gives 
$\kappa(M)^2 = (n+2t^2)(n+1+2t^2+t^4).$
We use Theorem \ref{thm:FCondNumber} to compute
$\kappa(F)^2$.  Note that since $c_0 = 1$, $\kappa(F)$ is independent of the initial step size of $F$.  Hence
\begin{eqnarray*}
\kappa(F) &=& ((n-1)+1+t^4+t^2+t^2) =  (n+2t^2+t^4).
\end{eqnarray*}
Thus we have
$$\frac{\kappa(F)^2}{\kappa(M)^2} = 
\frac{(n+2t^2+t^4)^2}{(n+2t^2)(n+1+2t^2+t^4)}.$$
The limit of the right hand side 
is $\frac{t^2}{2}$ as $t \rightarrow \infty$, which implies the final statement.
\end{proof}

  The following result
gives  another case where we can make
a matrix with smaller condition number
than any other Fielder companion matrix,
providing additional evidence that 
generalized companion matrices may be of interest.

\begin{theorem} 
Let $n \geq 2$, and fix $\ell \in \{3,\ldots,n-2\}$.
Let
$p(x) = x^n + c_{n-1}x^{n-1} + \cdots + c_1x + c_0$
with 
$c_0 = 1$, and
$(c_{\ell}c_{n-1})^2 < 2c_{\ell-1}c_{\ell}c_{n-1}-1$.
Let $M = M_n(c_\ell,\ell)$. 
Then $\kappa(M)<\kappa(F)$ for
every Fieldler companion matrix $F$ of $p(x)$.
\end{theorem}

\begin{proof}
Let $v = n - c_{\ell}^2 - c_{\ell-1}^2 + \sum_{i=1}^{n-1}c_i^2$.  
Because $c_0 = 1$,  by Theorem \ref{thm:FCondNumber} all
Fielder companion matrices $F$ have
condition number
$$\kappa(F)  = (v+c_{\ell}^2 + c_{\ell-1}^2).$$
By Lemma \ref{specialcondnumb}, 
with $a  = c_{\ell}$, 
\begin{eqnarray*}
\kappa(M)^2 &=& (v+c_{\ell}^2+(c_{\ell}c_{n-1}-c_{\ell-1})^2)
(v+c_{\ell}^2+c^2_{\ell-1}+1) \\
& = &(v+c_{\ell}^2+c^2_{\ell-1} + ((c_{\ell}c_{n-1})^2
-2c_{\ell-1}c_{\ell}c_{n-1}))(v+c_{\ell}^2+
c^2_{\ell-1}+1).
\end{eqnarray*}
If we set $w = (v+c_{\ell}^2 + c_{\ell-1}^2)$, then
$\kappa(M)^2=(w-y)(w+1)$
with $y=
2c_{\ell-1}c_{\ell}c_{n-1}-(c_{\ell}c_{n-1})^2.$
But $y>1$ by hypothesis, thus 
$\kappa(M)^2<w^2=\kappa(F)^2$.
%
\end{proof}


\end{document}